\documentclass[11pt]{article}
\usepackage{mathbbold}
\usepackage{amsfonts}
\usepackage{mathrsfs}
\usepackage{bbm}
\usepackage{latexsym,amsfonts,amssymb,amsmath,amsthm}
\usepackage{graphicx}

\usepackage{setspace}

\newtheorem{theorem}{Theorem}[section]
\newtheorem{definition}[theorem]{Definition}
\newtheorem{lemma}[theorem]{Lemma}

\newtheorem{prop}[theorem]{Proposition}
\numberwithin{equation}{section}
\def\O{\Omega}

\begin{document}

\baselineskip 20pt

\begin{center}

\textbf{\Large Singleton sets random attractor for stochastic
FitzHugh-Nagumo lattice equations driven by fractional Brownian
motions} \footnote{This work has been partially supported by NSFC
Grants 11071199, NSF of Guangxi Grants 2013GXNSFBA019008 and Guangxi
Provincial Department of Research Project Grants 2013YB102.}

\vskip 0.5cm

{\large  Anhui Gu, Yangrong Li}

\vskip 0.3cm

\textit{School of Mathematics and Statistics, Southwest
University, Chongqing 400715, China}\\

\vskip 1cm

\begin{minipage}[c]{15cm}

\noindent \textbf{Abstract}: The paper is devoted to the study of
the dynamical behavior of the solutions of stochastic
FitzHugh-Nagumo lattice equations, driven by fractional Brownian
motions, with Hurst parameter greater than $1/2$. Under some usual
dissipativity conditions, the system considered here features
different dynamics from the same one perturbed by Brownian motion.
In our case, the random dynamical system has a unique random
equilibrium, which constitutes a singleton sets random attractor.
\vspace{5pt}

\vspace{5pt}


\end{minipage}
\end{center}

\vspace{10pt}

\baselineskip 18pt

\section{Introduction}
Recently, the dynamics of deterministic lattice dynamical systems
have drawn much attention of mathematicians and physicists, see e.g.
\cite{BLW}-\cite{ZS} and the references therein. As we know, most of
the realistic systems involve noises which may play an important
role as intrinsic phenomena rather than just compensation of defects
in deterministic models. Stochastic lattice dynamical systems (SLDS)
arise naturally while these random influences or uncertainties are
taken into account. Since Bates et al. \cite{BLL} initiated the
study of SLDS, many works have been done regarding the existence of
global random attractors for SLDS with white noises on infinite
lattices (see e.g. \cite{CL}-\cite{WLX}). Later, the existence of
global random attractors was extended to other SLDS with additive
white noises, for example, first-order SLDS on $\mathbb{Z}^k$
\cite{LS1}, stochastic Ginzburg-Landau lattice equations \cite{LS2},
stochastic FitzHugh-Nagumo lattice equations \cite{Huang},
second-order stochastic lattice systems \cite{WLX} and first (or
second)-order SLDS with a multiplicative white noise \cite{CL,
Han1}. Zhao and Zhou \cite {ZZ1} gave some sufficient conditions for
the existence of a global random attractor for general SLDS in the
non-weighted space $\mathbb{R}$ of infinite sequences and provided
an application to damped sine-Gordon lattice systems with additive
noises. Very recently, Han et al. \cite{HSZ} provided some
sufficient conditions for the existence of global compact random
attractors for general SLDS in the weighted space $\ell_\rho ^p$ $
(p\geqslant 1)$ of infinite sequences, and their results are applied
to second-order SLDS in \cite{Han2} and \cite{Han3}.

However, as can be seen that all the work above are considered in
the framework of the classical It$\ddot{\rm o}$ theory of Brownian
motion. Note that fractional Brownian motion (fBm) does not possess
independent increments and stochastic differential equations driven
by fBm do not define the Markov process as in the case of usual
white noises. Therefore, it is not possible to apply standard
methods (see e.g. Theorem 3.1 in \cite{HSZ}) to deal with these
questions. Fortunately, the theory of random dynamical systems still
works for non-Markovian processes (see \cite{MS, GLS}). In
\cite{Gu2}, we consider the first-order lattice dynamical system
perturbed by fractional Brownian motions.

FBm appears naturally in the modeling of many complex phenomena in
applications when the systems are subject to ``rough" external
forcing. An fBm is a stochastic process which differs significantly
from the standard Brownian motion and semi-martingales, and other
classically used processes in the theory of stochastic process. As a
centered Gaussian process, it is characterized by the stationarity
of its increments and a medium- or long-memory property. It also
exhibits power scaling with exponent $H$. Its paths are H$\ddot{\rm
o}$lder continuous of any order $H'\in (0, H)$. An fBm is not a
semi-martingale nor a Markov process. Especially, when the Hurst
parameter $H\in (1/2, 1)$, the fBm has the properties of
self-similarity and long-range dependence. So, fBm is the good
candidate to model random long term influences in climate systems,
hydrology, medicine and physical phenomena. For more details on fBm,
we can refer to the monographs \cite{BHOZ, Mishura}.

Motivated by \cite{Huang, Gu2}, we investigate the long-term
behavior of the following stochastic FitzHugh-Nagumo lattice
equations:
\begin{equation}\label{1}
\left\{
\begin{array}{l}
\frac{du_i}{dt}=u_{i-1}-2u_i+u_{i+1}-\lambda u_i+f_i(u_i)-v_i
+a_i\frac{d\beta_i^H(t)}{dt},\\
\frac{dv_i}{dt}=\varrho u_i-\sigma v_i+b_i\frac{d\beta_i^H(t)}{dt},\\
 u(0)=u_{0}=(u_{i0})_{i\in \mathbb{Z}},\quad
 v(0)=v_{0}=(v_{i0})_{i\in \mathbb{Z}},
\end{array}
\right.
\end{equation}
where $\mathbb{Z}$ denotes the integer set, $u_i\in \mathbb{R} $,
$\lambda, \varrho$ and $\sigma$ are positive constants, $f_i$ are
smooth functions satisfying some dissipative conditions,
$(a_i)_{i\in \mathbb{Z}}\in \ell^2$, $(b_i)_{i\in \mathbb{Z}}\in
\ell^2$ and $\{\beta_i^H: i\in \mathbb{Z}\}$ are independent
two-sided fractional Brownian motions with Hurst parameter $H\in
(1/2, 1)$, $\ell^2=(\ell^2, ( \cdot , \cdot ), \|\cdot\|)$ denotes
the regular space of infinite sequences. When there are no noises
terms, form similar to \eqref{1} is the discrete of the
FitzHugh-Nagumo system which arose as modeling the signal
transmission across axons in neurobiology (see \cite{Jones}).
FitzHugh-Nagumo lattice system was used to stimulate the propagation
of action potentials in myelinated nerve axons (see \cite{EV}). The
stochastic FitzHugh-Nagumo lattice equations were first proposed in
\cite{Huang}. The existence of random attractors of (similar)
stochastic FitzHugh-Nagumo lattice equations with white noises were
established in \cite{Huang, WLW} and \cite{Gu1}.

The goal of this article is to establish the existence of a random
attractor for stochastic FitzHugh-Nagumo lattice equations with the
nonlinear $f$ under some dissipative conditions and driven by
fractional Brownian motions with Hurst parameter $H\in (1/2, 1)$. By
borrowing the main ideas of \cite{GKN}, we first define a random
dynamical system by using a pathwise interpretation of the
stochastic integral with respect to the fractional Brownian motions.
This method is based on the fact that a stochastic integral with
respect to an fBm with Hurst parameter $H\in (1/2, 1)$ can be
defined by a generalized pathwise Riemann-Stieltjes integral (see
e.g. \cite{Zahle}--\cite{TTV}). And then we show the existence of a
pullback absorbing set for the random dynamical system achieved by
means of a fractional Ornstein-Uhlenbeck transformation and Gronwall
lemma. Since every trajectory of the solutions of system \eqref{1}
cannot be differentiated, we have to consider the difference between
any two solutions among them, which is pathwise differentiable (see
\cite{GKN}). Due to the stationarity of the fractional
Ornstein-Uhlenbeck solution, we get a unique random equilibrium
finally. All solutions converge pathwise to each other, so the
random attractor, which consists of a unique random equilibrium, is
proven to be a singleton sets random attractor.

The paper is organized as follows. In Sec. 2, we recall some basic
concepts on random dynamical systems. In Sec. 3, we give a unique
solution to system \eqref{1} and make sure that the solution
generates a random dynamical system. We establish the main result,
that is, the random dynamical system generated by equation \eqref{1}
has a unique random equilibrium, which constitutes a singleton sets
random attractor in Sec. 4.

\section{Preliminaries}
In this section, we introduce some basic concepts related to random
dynamical systems and random attractors, which are taken from
\cite{Arnold}-\cite{CDF}.

Let $(\mathbb{E}, \|\cdot\|_{\mathbb{E}})$ be a separable Hilbert
space and $(\Omega, \mathcal{F}, \mathbb{P})$ be a probability
space.

\begin{definition} \label{MDS}
A metric dynamical system $(\Omega, \mathcal{F}, \mathbb{P},
\theta)$ with two-sided continuous time $\mathbb{R}$ consists of a
measurable flow
\begin{equation*}
\theta: (\mathbb{R}\times \Omega, \mathcal{B}(\mathbb{R})\otimes
\mathcal{F}) \rightarrow (\Omega, \mathcal{F}),
\end{equation*}
where the flow property for the mapping $\theta$ holds for the
partial mappings $\theta_t=\theta(t, \cdot)$:
\begin{equation*}
\theta_t\circ \theta_s=\theta_t\theta_s=\theta_{t+s}, \ \
\theta_0={\rm i d}_{\Omega}
\end{equation*}
for all $s, t\in \mathbb{R}$, and $\theta \mathbb{P}=\mathbb{P}$ for
all $t\in \mathbb{R}$.
\end{definition}

\begin{definition}
\label{RDS} A continuous random dynamical system (RDS) $\varphi $ on
$\mathbb{E}$ over $(\Omega ,\mathcal{F},\mathbb{P},(\theta _t)_{t\in
\mathbb{R}})$ is a $(\mathcal{B}(\mathbb{R}^{+})\times
\mathcal{F}%
\times \mathcal{B}(\mathbb{E}),\mathcal{B}(\mathbb{E}))$-measurable
mapping and satisfies

(i)  $\varphi (0,\omega )$ is the identity on $\mathbb{E}$;

(ii)  $\varphi (t+s,\omega )=\varphi (t,\theta _s\omega )\circ
\varphi (s,\omega )$ for all $s,$ $t\in \mathbb{R}^{+}$, $\omega \in
\Omega $;

(iii)  $\varphi (t,\omega )$ is continuous on $\mathbb{E}$ for all
$(t,\omega )\in \mathbb{R}^{+}\times \Omega $.
\end{definition}

A universe $\mathcal{D}=\{D(\omega), \omega \in \Omega\}$ is a
collection of nonempty subsets $D(\omega)$ of $\mathbb{E}$
satisfying the following inclusion property: if $D\in \mathcal{D}$
and $D'(\omega)\subset D(\omega)$ for all $\omega \in \Omega$, then
$D'\in \mathcal{D}$.

\begin{definition}\label{attractors}
A family $\mathcal{A}=\{A(\omega), \omega \in \Omega\}$ of nonempty
measurable compact subsets $\mathcal{A}(\omega)$ of $\mathbb{E}$ is
called $\varphi$- invariant if $\varphi(t, \omega,
\mathcal{A}(\omega))=\mathcal{A}(\theta_t\omega)$ for all $t\in
\mathbb{R^+}$ and is called a random attractor if in addition it is
pathwise pullback attracting in the sense that
$$H_d^*(\varphi(t, \theta_{-t}\omega, D(\theta_{-t}\omega)),
\mathcal{A}(\omega)) \rightarrow 0 \ \ \mbox{as} \ \
t\rightarrow\infty$$ for all $D\in \mathcal{D}$. Here $H_d^*$ is the
Hausdorff semi-distance on $\mathbb{E}$.
\end{definition}

\begin{definition}\label{equilibrium}
A random variable $u: \Omega\mapsto \mathbb{E}$ is said to be a
random equilibrium of the RDS $\varphi $ if it is invariant under
$\varphi $, i.e. if
\begin{equation*}
\varphi(t, \omega)u(\omega)=u(\theta_t\omega) \quad \mbox{for all}
\quad t\ge 0 \quad \mbox{and all} \quad \omega\in \Omega.
\end{equation*}
\end{definition}

\begin{definition}
A random variable $r: \Omega\rightarrow \mathbb{R}$ is called
tempered if
\begin{equation*}
\lim_{t\rightarrow\pm \infty}\frac{\log|r(\theta_t\omega)|}{|t|}=0 \
\ \mathbb{P}-a.s.
\end{equation*}
and a random set $\{D(\omega), \omega \in \Omega\}$ with
$D(\omega)\subset \mathbb{E}$ is called tempered if it is contained
in the ball $\{x\in \mathbb{R}: |x|\leq r(\omega)\}$, where $r$ is a
tempered random variable.
\end{definition}

Here we will always work with the attracting universe given by the
tempered random sets.

\begin{definition}
A family $\hat{B}=\{B\mathcal{(\omega)}, \omega \in \Omega\}$ is
said to be pullback absorbing if for every $D(\omega)\in
\mathcal{D}$, there exists $T_{D}(\omega)\geq 0$ such that
\begin{equation}\label{2}
\varphi(t, \theta_{-t}\omega, D(\theta_{-t}\omega))\subset B(\omega)
\ \ \forall t\geq T_{D}(\omega).
\end{equation}
\end{definition}

The following result (cf. Proposition 9.3.2 in \cite{Arnold},
Theorem 2.2 in \cite{CDF}) guarantees the existence of a random
attractor.

\begin{theorem}
\label{theorem 1} Let $(\theta, \varphi)$ be a continuous RDS on
$\Omega\times \mathbb{E}$. If there exists a pullback absorbing
family $\hat{B}=\{B\mathcal{(\omega)}, \omega \in \Omega\}$ such
that, for every $\omega \in \Omega$, $B(\omega)$ is compact and
$B(\omega)\in \mathcal{D}$, then the RDS $(\theta, \varphi)$ has a
random attractor

\[
\mathcal{A}(\omega )=\bigcap_{\tau>0}\overline{%
\bigcup_{t\geqslant \tau }\varphi (t,\theta _{-t}\omega )B(\theta
_{-t}\omega )}.
\]
\end{theorem}
Note that if the random attractor consists of singleton sets, i.e.
$\mathcal{A}(\omega)=\{u^*(\omega)\}$ for some random variable
$u^*$, then $u^*(t)(\omega)=u^*(t)(\theta_t\omega)$ is a stationary
stochastic process.

\section{FitzHugh-Nagumo Lattice Equations with
Fractional Brownian Motions}

We now recall the definition of a fractional Brownian motion. Given
$H\in (0, 1)$, a continuous centered Gaussian process $\beta^H(t),
t\in \mathbb{R}$, with the covariance function
\begin{equation*}
\mathbf{E}
\beta^H(t)\beta^H(s)=\frac{1}{2}(|t|^{2H}+|s|^{2H}-|t-s|^{2H}), \ \
t, s \in \mathbb{R}
\end{equation*}
is called a two-sided one-dimensional fBm, and $H$ is the Hurst
parameter. For $H=1/2$, $\beta$ is a standard Brownian motion, while
for $H\neq 1/2$, it is neither a semimartingale nor a Markov
process. Moreover,
\begin{equation*}
\mathbf{E}|\beta^H(t)-\beta^H(s)|^2=|t-s|^{2H}, \ \ \mbox{for all}\
\ s, t\in \mathbb{R}.
\end{equation*}

Here, we assume that $H\in (1/2, 1)$ throughout the paper. When
$H\in (0, 1/2)$ we cannot define the stochastic integral by a
generalized Stieljes integral and, therefore, dealing with such
values of the Hurst parameter seems to be much more complicated. It
is worth mentioning that when $H=1/2$ the fBm becomes the standard
Wiener process, the random dynamical system generated by the
(similar) stochastic FitzHugh-Nagumo lattice equations has been
studied in \cite{Huang, WLW}.

Using the definition of $\beta^H(t)$, Kolmogorov's theorem ensures
that $\beta^H$ has a continuous version, and almost all the paths
are H$\ddot{\rm o}$lder continuous of any order $H'\in (0, H)$ (see
\cite{Kunita}). Thus, let $\mathbb{E}=\ell^2\times\ell^2$ and norm
$\|\cdot\|_{\mathbb{E}}$, we can consider the canonical
interpretation of an fBm: denote $\Omega=C_0(\mathbb{R}, \ell^2)$,
the space of continuous functions on $\mathbb{R}$ with values in
$\ell^2$ such that $\omega(0)=0$, equipped with the compact open
topology. Let $\mathcal{F}$ be the associated Borel-$\sigma$-algebra
and $\mathbb{P}$ the distribution of the fBm $\beta^H$, and
$\{\theta_t\}_{t\in \mathbb{R}}$ be the flow of Wiener shifts such
that
\begin{equation*}
\theta_t\omega(\cdot)=\omega(\cdot+t)-\omega(t), \ \ t\in
\mathbb{R}.
\end{equation*}
Due to \cite{MS}-\cite{GS}, we know that the quadruple $(\Omega,
\mathcal{F}, \mathbb{P}, \theta)$ is an ergodic metric dynamical
system. Furthermore, it holds that
\begin{eqnarray}\label{3}
\begin{split}
&\beta^H(\cdot, \omega)=\omega(\cdot), \\
 \beta^H(\cdot, \theta_s\omega)&=\beta^H(\cdot+s, \omega)
-\beta^H(s, \omega)\\
&=\omega(\cdot+s)-\omega(s).
\end{split}
\end{eqnarray}
For $u=(u_i)_{i\in \mathbb{Z}}\in \ell^2$, define $\mathbb{A},
\mathbb{B}, \mathbb{B}^*$ to be linear operators from $\ell^2$ to
$\ell^2$ as follows:
\begin{eqnarray*}
\begin{split}
(\mathbb{A}u)_i&=-u_{i-1}+2u_i-u_{i+1}, \\
(\mathbb{B}u)_i&=u_{i+1}-u_i, \ \ (\mathbb{B}^*u)_i=u_{i-1}-u_i,\ \
i\in \mathbb{Z}. \end{split}
\end{eqnarray*}
It is easy to show that
$\mathbb{A}=\mathbb{B}\mathbb{B}^*=\mathbb{B}^*\mathbb{B}$, $(
\mathbb{B}^* u, u')=(u, \mathbb{B} u')$ for all $u, u'\in \ell^2$,
which implies that $(\mathbb{A}u, u)\geq 0$.

Let $W_1(t)\equiv W_1(t, \omega)=\sum_{i\in
\mathbb{Z}}a_i\omega_i(t)e^i$ and $W_2(t)\equiv W_2(t,
\omega)=\sum_{i\in \mathbb{Z}}b_i\omega_i(t)e^i$, here $(e^i)_{i\in
\mathbb{Z}}\in \ell^2$ denote the element having $1$ at position $i$
and the other components $0$. Then SLDS \eqref{1} with initial
conditions can be rewritten as pathwise Riemann-Stieltjes integral
equations in $\mathbb{E}$

\begin{eqnarray}\label{4}
\left\{
\begin{array}{l}
u(t)=u(0)+\int_0^t(-\mathbb{A}u(s)-\lambda u(s)+f(u(s))-v(s))ds
+W_1(t),\\
v(t)=v(0)+\int_0^t(\varrho u(s)-\sigma v(s))ds+W_2(t),\\
 u(0)=u_{0}=(u_{i0})_{i\in \mathbb{Z}},\quad
 v(0)=v_{0}=(v_{i0})_{i\in \mathbb{Z}},
\end{array}
\right.
\end{eqnarray}
where $u=(u_i)_{i\in \mathbb{Z}}$, $\lambda, \varrho$ and $\sigma$
are positive constants, $a=(a_i)_{i\in \mathbb{Z}}\in \ell^2$,
$b=(b_i)_{i\in \mathbb{Z}}\in \ell^2$ and $\{\omega_i=\beta^H_i:
i\in \mathbb{Z}\}$ are independent two-sided fractional Brownian
motions with Hurst parameter $H\in (1/2, 1)$, $f(u)=(f_i(u_i))_{i\in
\mathbb{Z}}$ is a nonlinear smooth function satisfies a one-sided
dissipative Lipschitz condition

\begin{equation}\label{5}
(f(u)-f(v), u-v)\le -\gamma \|u-v\|^2\ \ \text{for all} \ u, v\in
\mathbb{R}
\end{equation}
and the polynomial growth condition
\begin{equation}\label{6}
|f(u)|\le c_f(|u|^{2p+1}+1) \ \text{for all} \ \ u\in \mathbb{R},
\end{equation}
where $\gamma$ is a positive constant, $p$ is a positive integer.

In addition we could consider a more general dissipativity
condition, which would lead to nontrivial setvalued random
attractors, we will restrict here to the dissipativity condition
\eqref{5}. When system \eqref{1} with Hurst parameter $H=1/2$ and
under conditions \eqref{5} and \eqref{6}, we can apply the result of
Theorem 3.1 in \cite{HSZ}, i.e. the combination of the existence of
a bounded closed random absorbing set and the property of random
asymptotic nullity to get the existence of a compact random
attractor. Moreover, we have the following results:

\begin{lemma} \label{estimate}
There exists positive random constants
$(\tilde{\rho}_i(\omega))_{i\in \mathbb{Z}}\in \ell^2$ and
$\rho(\omega)=\|\tilde{\rho}(\omega)\|$ such that for every
$\omega\in \bar{\Omega}$, where $\bar{\Omega}\in \mathcal{F}$ is a
$(\theta_t)_{t\in \mathbb{R}}$-invariant set of full measure, the
fractional Brownian motions are well defined for $t\in \mathbb{R}$
in $\ell^2$ satisfying

\begin{eqnarray*}
\|W_j(t)\|^2\le 2\max\{\|a\|^2, \|b\|^2\}\rho^2(\omega)(1+|t|^4), \
\ j=1,2.
\end{eqnarray*}
\end{lemma}
\begin{proof}
Obviously.
\end{proof}

\begin{prop}\label{existence}
Let the above assumptions on $f$ be satisfied and $T>0$. Then system
\eqref{4} has a unique pathwise solution $\Psi=(\Psi(t))_{t\geq
0}=(u(t), v(t))_{t\geq 0}$. Furthermore, the solution satisfies
\begin{eqnarray*}
&&\sup_{t\in [0, T]}\|\Psi(t)\|_{\mathbb{E}}^2 \le
M[\|\Psi_0\|_{\mathbb{E}}^2+\sup_{t\in [0,
T]}(\|W_1(t)\|^2+\|W_2(t)\|^2)\nonumber\\
&& ~~~~~~~~~~~~~~~+\int_0^T(\|W_1(s)\|^{4p+2}
+\|W_1(s)\|^2+\|W_2(s)\|^2+1)ds],
\end{eqnarray*}
where $M$ is a positive constant independent of $T$.
\end{prop}

\begin{proof}
Let $\tilde{u}(t)=u(t)-W_1(t)$ and $\tilde{v}(t)=v(t)-W_2(t)$,
system \eqref{4} has a solution $\Psi=(\Psi(t))_{t\geq 0}$ for all
$\omega\in \Omega$ if and only if the following system
\begin{eqnarray}\label{8}
\left\{
\begin{array}{l}
\tilde{u}(t)=u(0)+\int_0^t(-\mathbb{A}\tilde{u}(s)-\lambda
\tilde{u}(s)+f(\tilde{u}(s)+W_1(s))\\
~~~~~~~~~~~~~~~~-\tilde{v}(s)-\mathbb{A}W_1(s)-\lambda
W_1(s)-W_2(s))ds,\\
\tilde{v}(t)=v(0)+\int_0^t(\varrho \tilde{u}(s)
-\sigma \tilde{v}(s)-\varrho W_1(s)-\sigma W_2(s))ds,\\
 u(0)=u_{0}=(\tilde{u}_{i0})_{i\in \mathbb{Z}},\quad
 v(0)=v_{0}=(\tilde{v}_{i0})_{i\in \mathbb{Z}}
\end{array}
\right.
\end{eqnarray}
has a unique pathwise solution for $t\in [0, T]$. However, since the
integrand is pathwise continuous, the fundamental theorem of
calculus says that the left hand side of \eqref{8} is pathwise
differentiable. Thus, for a fixed $\omega\in \Omega$, system
\eqref{8} is the pathwise system of random ODEs
\begin{eqnarray}\label{9}
\left\{
\begin{array}{l}
\frac{d\tilde{u}(t)}{dt}=-\mathbb{A}\tilde{u}(t)-\lambda
\tilde{u}(t)+f(\tilde{u}(t)+W_1(t))\\
~~~~~~~~~~~~~~~~-\tilde{v}(t)-\mathbb{A}W_1(t)-\lambda
W_1(t)-W_2(t),\\
\frac{d\tilde{v}(t)}{dt}=\varrho \tilde{u}(t)
-\sigma \tilde{v}(t)+\varrho W_1(t)-\sigma W_2(t),\\
 u(0)=u_{0}=(\tilde{u}_{i0})_{i\in \mathbb{Z}},\quad
 v(0)=v_{0}=(\tilde{v}_{i0})_{i\in \mathbb{Z}}.
\end{array}
\right.
\end{eqnarray}
Since $f(u)$ is a continuous function, and the assumptions on $f$
are satisfied, by the standard argument on existence theorem for
ODEs, it follows that system \eqref{9} possesses a local solution in
a small interval $[0, \tau(\omega)]$, which means system \eqref{4}
has a unique local solution in the same small interval $[0,
\tau(\omega)]$. Here, we remain to show that the local solution is a
global one.

For a fixed $\omega\in \Omega$, by taking the inner product of
\eqref{9} with $(\tilde{u}, \tilde{v})$ in $\mathbb{E}$, it follows
that
\begin{eqnarray}\label{10}
&&\|\tilde{u}(t)\|^2+\frac{1}{\varrho}\|\tilde{v}(t)\|^2=
\|\tilde{u}_0\|^2+\frac{1}{\varrho}\|\tilde{v}_0\|^2+
2\int_0^t(-\mathbb{A}\tilde{u}(s),\tilde{u}(s))ds\nonumber\\
&&~~+2\int_0^t(f(\tilde{u}(s)+W_1(s)), \tilde{u}(s))ds
+2\int_0^t(-\mathbb{A}W_1(s), \tilde{u}(s))ds\nonumber\\
&&~~~-2\int_0^t(\lambda W_1(s), \tilde{u}(s))ds-2\int_0^t(W_2(s),
\tilde{u}(s))ds\nonumber\\
&&~~~~-2\lambda\int_0^t\|\tilde{u}(s)\|^2ds
-\frac{2\sigma}{\varrho}\int_0^t\|\tilde{v}(s)\|^2ds\nonumber\\
&&~~~~~~-\frac{2\sigma}{\varrho}\int_0^t(W_2(s), \tilde{v}(s))ds
+2\int_0^t(W_1(s), \tilde{v}(s))ds.
\end{eqnarray}
By \eqref{5} and \eqref{6}, we obtain that
\begin{eqnarray}\label{11}
&&2(f(\tilde{u}(s)+W_1(s)),
\tilde{u}(s))\nonumber\\
&=&2(f(\tilde{u}(s)+W_1(s)), \tilde{u}(s)+W_1(s))
-2(f(\tilde{u}(s)+W_1(s), W_1(s))\nonumber\\
&\le& -\gamma \|\tilde{u}(s)+W_1(s)\|^2+2|f(\tilde{u}(s)
+W_1(s)||W_1(s)|\nonumber\\
&\le& c_1(\|W_1(s)\|^{4p+2}+\|W_1(s)\|^2+1),
\end{eqnarray}
where $c_1$ is a positive constant depends on $\gamma, c_f$ and $p$.
By Young's inequality, it yields that
\begin{eqnarray}\label{12}
&&2(-\mathbb{A}\tilde{u}(s),\tilde{u}(s))+2(f(\tilde{u}(s)+W_1(s)),
\tilde{u}(s))+2(-\mathbb{A}W_1(s), \tilde{u}(s))\nonumber\\
&\le& \lambda \|\tilde{u}(s)\|^2+c_2(\|W_1(s)\|^2+\|W_2(s)\|^2),
\end{eqnarray}
where $c_2$ is a positive constant depends on $\lambda$, and

\begin{eqnarray}\label{13}
&&\frac{2\sigma}{\varrho}(W_2(s), \tilde{v}(s))
+2(W_1(s), \tilde{v}(s))\nonumber\\
&\le& \frac{\sigma}{\varrho}\|\tilde{v}(s)\|^2
+c_3(\|W_1(s)\|^2+\|W_2(s)\|^2),
\end{eqnarray}
where $c_3$ is a positive constant depends on $\varrho$ and
$\sigma$. Let $\alpha=\min\{\lambda, \sigma\}$ and combine
\eqref{11}-\eqref{13}with \eqref{10}, for $t\ge0$, we get
\begin{eqnarray}\label{14}
&&\|\tilde{u}(t)\|^2+\frac{1}{\varrho}\|\tilde{v}(t)\|^2\le
\|\tilde{u}_0\|^2+\frac{1}{\varrho}\|\tilde{v}_0\|^2-\alpha
\int_0^t(\|\tilde{u}(s)\|^2
+\frac{1}{\varrho}\|\tilde{v}(s)\|^2)ds\nonumber\\
&&~~~~~~~~~~~~~+c\int_0^t(\|W_1(s)\|^{4p+2}
+\|W_1(s)\|^2+\|W_2(s)\|^2+1)ds\nonumber\\
&&~~~~~~~~~~~\le
\|\tilde{u}_0\|^2+\frac{1}{\varrho}\|\tilde{v}_0\|^2\nonumber\\
&&
~~~~~~~~~~~~~~~+c\int_0^t(\|W_1(s)\|^{4p+2}+\|W_1(s)\|^2+\|W_2(s)\|^2+1)ds,
\end{eqnarray}
where $c$ is a positive constant depends on $\varrho, \sigma,
\lambda, \gamma, c_f$ and $p$. Hence, from \eqref{14}, we know that
$\|\tilde{u}(t)\|^2+\frac{1}{\varrho}\|\tilde{v}(t)\|^2$ is bounded
by a continuous function, which implies the global existence of a
solution on interval $[0, T]$. Furthermore, for all $\omega\in
\Omega$, it follows that
\begin{eqnarray}\label{15}
&&\sup_{t\in [0,
T]}(\|u(t)\|^2+\frac{1}{\varrho}\|v(t)\|^2)=\sup_{t\in [0,
T]}(\|\tilde{u}(t)+W_1(t)\|^2
+\frac{1}{\varrho}\|\tilde{v}(t)+W_2(t)\|^2)
\nonumber\\
&&~~~~~~~~~~~\le 2(\|u_0\|^2
+\frac{1}{\varrho}\|v_0\|^2)+2\sup_{t\in [0,
T]}(\|W_1(t)\|^2+\frac{1}{\varrho}\|W_2(t)\|^2)\nonumber\\
&&
~~~~~~~~~~~~~~~+2c\int_0^T(\|W_1(s)\|^{4p+2}+\|W_1(s)\|^2+\|W_2(s)\|^2+1)ds.
\end{eqnarray}
According to Lemma \ref{estimate}, we know that the right side of
\eqref{15} is well defined. Let $\tilde{\alpha}=\frac{\max\{1,
\frac{1}{\varrho}\}}{\min\{1, \frac{1}{\varrho}\}}$, the proof is
complete.
\end{proof}

\begin{prop} \label{CRDS}
The solution of \eqref{4} determinants a continuous random dynamical
system $\varphi: \mathbb{R^+}\times \Omega\times
\mathbb{E}\rightarrow \mathbb{E}$, which is given by
\begin{eqnarray}\label{19}
\varphi(t, \omega, \Psi_0)=\Psi_0+\int_0^tG(\Psi(s))ds+\eta(t,
\omega) \ \ \mbox{for} \ \ t\geq 0,
\end{eqnarray}
where $G(\Psi(t))=L\Psi(t)+F(\Psi(t))$ and
\begin{eqnarray*}
L=\left ( \begin{array}{cc}
     -\mathbb{A}-\lambda & \ \ -1\\
     \varrho&
     \ \ -\sigma
\end{array} \right),\
F(\Psi)=\left ( \begin{array}{c}
     f(u)\\
    0
\end{array} \right), \ \eta(t, \omega)=\left ( \begin{array}{c}
     W_1(t, \omega)\\
     W_2(t, \omega)
\end{array} \right).
\end{eqnarray*}
\end{prop}

\begin{proof}
For the need of making the relations clear between
$G(\Psi(t))(\cdot)$ and $\omega$, we write $G(\Psi(t))(\omega)$
instead if necessary. Note that \eqref{3} is satisfied for
$\omega\in \Omega$ and by the definition of $(\theta_t)_{t\in
\mathbb{R}}$, we have the property
\begin{equation*}
\eta(\tau+t, \omega)=\eta(\tau, \theta_{t}\omega)+\eta(t, \omega) \
\ \mbox{for all} \ \ t, \tau\in \mathbb{R}.
\end{equation*}
By Proposition \ref{existence} we know that $\varphi$ solves
\eqref{4}, thus $\varphi$ is measurable and satisfies $\varphi(0,
\omega, \cdot)=\rm id_{\mathbb{E}}$. It remains to verify that the
cocycle property in Definition \ref{RDS}. Let $t, \tau\in
\mathbb{R^+}, \omega\in \Omega$ and $\Psi_0\in \mathbb{E}$, it
yields from \eqref{3} that
\begin{eqnarray*}
&&\varphi(t+\tau, \omega, \Psi_0)\nonumber\\
&=&\Psi_0+\int_0^{t+\tau} G(\Psi(s))(\omega)ds+\eta(t+\tau,
\omega)\nonumber\\
&=& \Psi_0+\int_0^tG(\Psi(s))(\omega)ds+\eta(t, \omega)+\int_t^{t
+\tau}G(\Psi(s))(\omega)ds+\eta(\tau,
\theta_t\omega)\nonumber\\
&=& \Psi(t)+\int_0^{\tau} G(\Psi(s)) (\theta_t\omega)ds+\eta(\tau,
\theta_t\omega)\nonumber\\
&=& \varphi(\tau, \theta_{t}\omega, \cdot)\circ \varphi(t, \omega,
\Psi_0),
\end{eqnarray*}
which completes the proof.
\end{proof}

\section{Existence of a Random Attractor}

In this section, we will prove the existence of a random attractor
for the RDS defined in Proposition \ref{CRDS}. Sometimes, for the
need of making the relations between $\bar{u}(\cdot)$ (or $\bar{v}$,
$\Psi$, $\bar{\Phi}$) and $\omega$ more explicitly, we will write
$\bar{u}(\omega)$ (or $\bar{v}(\omega)$, $\Psi(\omega)$,
$\bar{\Phi}(\omega)$) instead if necessary.

Consider the following fractional Ornstein-Uhlenbeck processes
\begin{equation}\label{21}
du(t)=-\lambda u(t)dt+dW_1(t), \ dv(t)=-\sigma v(t)dt+dW_2(t),
\end{equation}
where $\lambda, \sigma$ defined in \eqref{4} and $W_1(t), W_2(t)$
denote one-dimensional fractional Brownian motions. They have the
explicit solutions
\begin{equation}\label{22}
u(t)=u_0e^{-\lambda t}+e^{-\lambda t}\int_0^t e^{\lambda s}dW(s), \
v(t)=v_0e^{-\sigma t}+e^{-\sigma t}\int_0^t e^{\sigma s}dW(s).
\end{equation}
Take the pathwise pullback limits, we get the stochastic stationary
solutions
\begin{equation}\label{23}
\bar{u}(t)=e^{-\lambda t}\int_{-\infty}^te^{\lambda s}dW(s),\
\bar{v}(t)=e^{-\sigma t}\int_{-\infty}^te^{\sigma s}dW(s), \ \ t\in
\mathbb{R},
\end{equation}
which are called the fractional Ornstein-Uhlenbeck solutions. We
have the following properties:

\begin{lemma} \label{property}
There exists positive random constants
$(\check{\rho}_i(\omega))_{i\in \mathbb{Z}},
(\hat{\rho}_i(\omega))_{i\in \mathbb{Z}}\in \ell^2$ and
$\check{\rho}^2(\omega)=16\sum_{i\in
\mathbb{Z}}a_i^2\check{\rho}_i^2(\omega),
\hat{\rho}^2(\omega)=16\sum_{i\in
\mathbb{Z}}b_i^2\hat{\rho}_i^2(\omega)$ for all $\omega\in \Omega$,
the Riemann-Stieltjes integrals in \eqref{23} are well defined in
$\ell^2$. Moreover, for all $\omega\in \Omega, t\in \mathbb{R}$, we
have
\begin{equation*}
\|e^{-\lambda t}\int_{-\infty}^te^{\lambda s} dW_1(s)\|\leq
\check{\rho}(\omega)(1+|t|)^2, \ \|e^{-\sigma
t}\int_{-\infty}^te^{\sigma s} dW_2(s)\|\leq
\hat{\rho}(\omega)(1+|t|)^2.
\end{equation*}
\end{lemma}
\begin{proof}
By the Lemma 1 in \cite{GKN}, we can easily get the conclusion.
\end{proof}

Now, we are in the position to state the main result.
\begin{theorem}
Assume that the conditions on $f$ are satisfied. Then the random
dynamical system $\varphi$ has a unique random equilibrium, which
constitutes a singleton sets random attractor.
\end{theorem}

\begin{proof}

Let $\Psi(t)=(u(t), v(t)), ~\Phi(t)=(\tilde{u}(t), \tilde{v}(t)) $
be any two solutions of system \eqref{1}. Their sample paths are not
differentiable, but the difference satisfies pathwise for $t\geq 0$,
\begin{eqnarray*}
\Psi(t)-\Phi(t)=\Psi_0-\Phi_0+\int_0^t(L(\Psi(s)-\Phi(s))
+(F(\Psi(s))-F(\Phi(s)))ds,
\end{eqnarray*}
and again, since the integrand is pathwise continuous, the
fundamental theorem of calculus indicates that the left hand side is
pathwise differentiable and satisfies
\begin{eqnarray}\label{19}
&&\frac{d}{dt}(\Psi(t)-\Phi(t))=L(\Psi(t)-\Phi(t))+F(\Psi(t))
-F(\Phi(t)), \ t\geq 0.
\end{eqnarray}
Recall that $\alpha=\min\{\lambda, \sigma\}$, we obtain from
\eqref{19} that
\begin{eqnarray}
\begin{split}
\frac{d}{dt}\|\Psi(t)-\Phi(t)\|_{\mathbb{E}}^2&=
2( \Psi(t)-\Phi(t), L(\Psi(t)-\Phi(t)))_{\mathbb{E}}\nonumber\\
&\quad\quad+2(\Psi(t)-\Phi(t), F(\Psi(t))
-F(\Phi(t)))_{\mathbb{E}}\nonumber\\
&\leq -2\alpha\|\Psi(t)-\Phi(t)\|_{\mathbb{E}}^2.
\end{split}\label{20}
\end{eqnarray}
Thus pathwise we have
\begin{equation*}
\|\Psi(t)-\Phi(t)\|_{\mathbb{E}}^2\leq \|\Psi_0-\Phi_0
\|_{\mathbb{E}}^2e^{-2\alpha t}\rightarrow 0, \ \ \mbox{as}
 \ \ t\rightarrow\infty.
\end{equation*}
That is to say that all solutions converge pathwise forward to each
other in time.

Now, we want to know where the solution will converge to. Let
$\bar{\Phi}(t)=(\bar{u}(t), \bar{v}(t))$. We consider the difference
$\Psi(t)-\bar{\Phi}(t)$. Since their paths are continuous, the
difference is pathwise differentiable and satisfies the integral
equation for $t\geq 0$,
\begin{eqnarray*}
&&\Psi(t)-\bar{\Phi}(t)=\Psi_0-\bar{\Phi}_0
+\int_0^t(L(\Psi(s)-\bar{\Phi}(s))+(F(\Psi(s))-F(\bar{\Phi}(s)))ds,
\end{eqnarray*}
which is equivalent to the pathwise differential equation
\begin{eqnarray*}
\frac{d}{dt}(\Psi(t)-\bar{\Phi}(t))=L(\Psi(s)-\bar{\Phi}(s))
+(F(\Psi(s))-F(\bar{\Phi}(s)), \ \ t\geq 0.
\end{eqnarray*}
That is to consider the following system
\begin{eqnarray}\label{25}
\left\{
\begin{array}{l}
\frac{d}{dt}(u(t)-\bar{u}(t))=-\mathbb{A}u(t)-\lambda
(u(t)-\bar{u}(t))+f(u(t))-v(t),\\
\frac{d}{dt}(v(t)-\bar{v}(t))=\varrho u(t) -\sigma
(v(t)-\bar{v}(t)).
\end{array}
\right.
\end{eqnarray}
By taking the inner product in $\mathbb{E}$, we get
\begin{eqnarray}\label{26}
&&\frac{d}{dt}(\|u(t)-\bar{u}(t)\|^2+\frac{1}{\varrho}\|v(t)
-\bar{v}(t)\|^2)\nonumber\\
&=& 2(-\mathbb{A}u(t), u(t)-\bar{u}(t))+2(f(u), u(t)
-\bar{u}(t))\nonumber\\
&&\quad-2(v(t),  u(t)-\bar{u}(t))+2(u(t),  v(t)-\bar{v}(t))\nonumber\\
&&\quad\quad-2\lambda\|u(t)-\bar{u}(t)\|^2
-\frac{2\sigma}{\varrho}\|v(t)-\bar{v}(t)\|^2.
\end{eqnarray}
We know that
\begin{eqnarray*}
&&2(-\mathbb{A}u(t),
u(t)-\bar{u}(t))=2(-\mathbb{A}(u(t)-\bar{u}(t)), u(t)-\bar{u}(t))\\
&&~~~~~~~~~~~~~~~~~~~~~+2(\mathbb{A}\bar{u}(t), u(t)-\bar{u}(t))\\
&&\le
\frac{\lambda}{2}\|u(t)-\bar{u}(t)\|^2+\frac{32}{\lambda}\|\bar{u}(t)\|^2,
\end{eqnarray*}
\begin{eqnarray*}
&&2(f(u), u(t)-\bar{u}(t))=2(f(u)-f(\bar{u}), u(t)-\bar{u}(t))\\
&&~~~~~~~~~~~~~~+2(f(\bar{u}), u(t)-\bar{u}(t))\\
&&\le-\gamma\|u(t)-\bar{u}(t)\|^2+\frac{\lambda}{2}\|u(t)-\bar{u}(t)\|^2
+\frac{8}{\lambda}\|f(\bar{u})\|^2,
\end{eqnarray*}
\begin{eqnarray*}
&&-2(v(t),  u(t)-\bar{u}(t))+2(u(t),
v(t)-\bar{v}(t))\\
&\le& \gamma\|u(t)-\bar{u}(t)\|^2+\frac{4}{\gamma}
\|\bar{v}(t)\|^2+\frac{\sigma}{\varrho}\|v(t)-\bar{v}(t)\|^2
+\frac{4\varrho}{\sigma}\|\bar{u}(t)\|^2.
\end{eqnarray*}
Combine the three inequalities above with \eqref{26}, we have
\begin{eqnarray}\label{27}
&&\frac{d}{dt}(\|u(t)-\bar{u}(t)\|^2+\frac{1}{\varrho}\|v(t)
-\bar{v}(t)\|^2)\nonumber\\
&\le&-\lambda\|u(t)-\bar{u}(t)\|^2-\frac{\sigma}{\varrho}
\|v(t)-\bar{v}(t)\|^2\nonumber\\
&&\quad+c_4(\|\bar{u}(t)\|^2+\|\bar{v}(t)\|^2+\|f(\bar{u})\|^2),
\end{eqnarray}
where $c_4$ is a positive constant depends on $\lambda, \varrho$ and
$\sigma$. Then we obtain
\begin{eqnarray*}
\frac{d}{dt}\|\Psi(t)-\bar{\Phi}(t)\|_{\mathbb{E}}^2\le
-\alpha\|\Psi(t)-\bar{\Phi}(t)\|_{\mathbb{E}}^2
+c_4(\|\bar{u}(t)\|^2+\|\bar{v}(t)\|^2+\|f(\bar{u})\|^2),
\end{eqnarray*}
and hence
\begin{eqnarray}\label{29}
&&\|\Psi(t)-\bar{\Phi}(t)\|_{\mathbb{E}}^2\le
\|\Psi_0(\omega)-\bar{\Phi}_0(\omega)\|^2e^{-\alpha t}\nonumber\\
&&~~~~~~~~~~~+c_4e^{-\alpha t}\int_0^te^{\alpha s}
(\|\bar{u}(s)\|^2+\|\bar{v}(s)\|^2+\|f(\bar{u}(s))\|^2)ds.
\end{eqnarray}
Let us check that the family of balls centered on
$\bar{\Phi}_0(\omega)$ with the random radius
\begin{equation}\label{33}
R(\omega):=\sqrt{1+c_4\int_{-\infty}^0e^{\alpha  s}
(\|\bar{u}(s)(\omega)\|^2+\|\bar{v}(s)(\omega)\|^2
+\|f(\bar{u}(s)(\omega))\|^2)ds}
\end{equation}
is a pullback absorbing family for the random dynamical system
generated by system \eqref{1}.

Due to the assumptions on $f$ and Lemma \ref{property}, the radius
defined in \eqref{33} is well defined. Now, by replacing $\omega$ by
$\theta_{-t}\omega$ in \eqref{29}, we get
\begin{eqnarray}\label{34}
&&\|\Psi(\theta_{-t}\omega)-\bar{\Phi}(\theta_{-t}\omega)
\|_{\mathbb{E}}^2\nonumber\\
&\leq& \|\Psi_0(\theta_{-t}\omega)-\bar{\Phi}_0
(\theta_{-t}\omega)\|_{\mathbb{E}}^2e^{-\alpha  t}\nonumber\\
&&+c_4\int_0^te^{\alpha  (s-t)}(\|\bar{u}(s)
(\theta_{-t}\omega)\|^2+\|\bar{v}(s) (\theta_{-t}\omega)\|^2
+\|f(\bar{u}(s)(\theta_{-t}\omega))\|^2)ds\nonumber\\
&=& \|\Psi_0(\theta_{-t}\omega)-\bar{\Phi}_0
(\theta_{-t}\omega)\|_{\mathbb{E}}^2e^{-\alpha  t}\nonumber\\
&&+c_4\int_{-t}^0e^{\alpha  s}(\|\bar{u}(s)(\omega)
\|^2+\|\bar{v}(s)(\omega)\|^2 +\|f(\bar{u}(s)(\omega))\|^2)ds.
\end{eqnarray}
The last term in \eqref{34} due to $\bar{u}(s)(\theta_{-t}\omega)=
\bar{u}_0(\theta_{s-t}\omega)=\bar{u}(s-t)(\omega)$ and
$\bar{v}(s)(\theta_{-t}\omega)=
\bar{v}_0(\theta_{s-t}\omega)=\bar{v}(s-t)(\omega)$ which deduced
from that $(\bar{u}(t))_{t\in \mathbb{R}}$ and $(\bar{v}(t))_{t\in
\mathbb{R}}$ are stationary processes. The conclusion now follows as
$t\rightarrow\infty$.

Because of the stationarity and Lemma \ref{property}, we have
$e^{-\alpha
t}\|\bar{\Phi}_0(\theta_{-t}\omega)\|_{\mathbb{E}}^2=e^{-\alpha
t}\|\bar{\Phi}(-t)(\omega)\|_{\mathbb{E}}^2 \rightarrow 0$ as
$t\rightarrow\infty$. Then we have the pullback absorption
\begin{equation}\label{35}
\|\Psi(\theta_{-t}\omega)\|_{\mathbb{E}}^2\leq
\|\bar{\Phi}_0(\omega)\|_{\mathbb{E}}^2+R^2(\omega), \ \ \forall
t\geq T_{\mathcal{D}(\omega)}.
\end{equation}
So, we have the stationary random process
$\tilde{\Phi}(t)(\omega):=\tilde{\Phi}_0(\theta_t\omega)$, which
pathwise attracts all other solutions in both forward and pullback
senses, is a random equilibrium. Now, we define a singleton sets
$\mathcal{A}=\{A(\omega), \omega\in
\Omega\}=\{\tilde{\Phi}_0(\omega)\}$, i.e. the singleton sets is
formed by the random equilibrium. Here, we remain to show that the
singleton sets turns out to be a random attractor. According to
Definition \ref{attractors}, we can easily get the compactness,
invariance and attraction (implied by absorbtion). The proof is
complete.

\end{proof}

\section{Conclusions}
We studied the stochastic FitzHugh-Nagumo equations driven by
fractional Brownian motion. The existence of the random attractor
formed by the unique random equilibrium turns out to be a single
sets random attractor, which differs from the results obtained in
\cite{Huang} and \cite{WLW} where the same system is driven by white
noises. The methodology can be used to deal with other stochastic
lattice systems, which is a topic that will be the focus of further
research.

\vskip 1cm

\section*{Acknowledgements}
The authors would like to express their sincere thanks to the
anonymous referees for their time and helpful comments and
suggestions, which have largely improved the presentation of this
paper.

\end{document}